\documentclass[smallextended]{svjour3} 

\usepackage{amssymb,amsmath,color,verbatim}
\usepackage[latin1]{inputenc}
\usepackage{float}
\usepackage{graphicx}
\usepackage{epstopdf}

\smartqed

\def\zmod#1{\,\,({\rm mod}\,\,#1)}

\newcommand{\F}{\mathbb{F}}

\newcommand{\Q}{\mathbb{Q}}

\newcommand{\Z}{\mathbb{Z}}

\newcommand{\C}{\mathbb{C}}
\newcommand{\mce}{\Or_{E}}
\newcommand{\mcf}{\Or_{F}}

\newcommand{\mcl}{\Or_{L}}
\newcommand{\mcc}{\mathcal{C}}
\newcommand{\G}{{\rm Gal}}

\newcommand{\Or}{\mathcal{O}}

\newcommand{\rank}{{\rm rank}}

\DeclareMathOperator{\frob}{Frob}
\DeclareMathOperator{\Cl}{Cl}
\DeclareMathOperator{\Mat}{Mat}
\DeclareMathOperator{\rk}{rank}
\DeclareMathOperator{\Nm}{Nm}

\DeclareMathOperator{\SNR}{SNR}    		
\DeclareMathOperator{\nat}{nat}    	 
\DeclareMathOperator{\disc}{disc}   		
\DeclareMathOperator{\nr}{nr}
\DeclareMathOperator{\tr}{tr}
\DeclareMathOperator{\Tr}{Tr}

\begin{document}

\title{Natural orders for asymmetric space--time coding: minimizing the discriminant}

\author{Amaro Barreal \and Capi~Corrales~Rodrig\'{a}\~{n}ez \and Camilla Hollanti}

\institute{A. Barreal \at
           Department of Mathematics and Systems Analysis, Aalto University, Finland \\
           Tel.: +358-50-5935194 \\
           \email{amaro.barreal@aalto.fi}           
           \and
           C. Corrales Rodrig\'{a}\~{n}ez \at
           Faculty of Mathematical Sciences, Complutense University of Madrid, Spain 
           \and
           C. Hollanti \at 
		   Department of Mathematics and Systems Analysis, Aalto University, Finland  
		  }

\maketitle 

\begin{abstract}
Algebraic space--time coding --- a powerful technique developed in the context of multiple-input multiple-output (MIMO) wireless communications --- has profited tremendously from tools from Class Field Theory and, more concretely, the theory of central simple algebras and their orders. During  the last decade, the study of space--time codes for practical applications, and more recently for future generation (5G+) wireless systems, has provided a practical motivation for the consideration of many interesting mathematical problems. 
One such problem is the explicit computation of orders of central simple algebras with small discriminants. In this article, we consider the most interesting asymmetric MIMO channel setups and, for each treated case, we provide explicit pairs of fields and a corresponding non-norm element giving rise to a cyclic division algebra whose natural order has the minimum possible discriminant. 

\keywords{Central Simple Algebras \and Division Algebras \and Discriminant \and Natural Orders \and MIMO \and Space--Time Coding.}
\end{abstract}

\section{Introduction}
\label{sec:intro}

The existing contemporary communications systems can be abstractly characterized by the conceptual seven-layer Open Systems Interconnection model. The lowest (or first) layer, known as the \emph{physical layer}, aims to describe the communication process over an actual physical medium. Due to the increasing demand for flexibility, information exchange nowadays often occurs via antennas at the transmitting and receiving end of a wireless medium, \emph{e.g.}, using mobile phones or tablets for data transmission and reception.  
An electromagnetic signal transmitted over a wireless channel is however prone to interference, fading, and environmental effects caused by, \emph{e.g.}, surrounding buildings, trees, and vehicles, making reliable wireless communications a challenging technological problem. 

With the advances in communications engineering, it was soon noticed that increasing the number of spatially separated transmit and receive antennas, as well as adding redundancy by repeatedly transmitting the same information encoded over multiple time instances\footnote{'Time instances' are commonly referred to as \emph{channel uses}.}, can dramatically improve the transmission quality. A code representing both diversity over time and space is thus called a \emph{space--time code}. Let us consider a channel with $n_t$ and $n_r$ antennas at its transmitting and receiving end, respectively, and assume that transmission occurs over $T$ consecutive time instances. If $n_t = n_r$, the channel is called \emph{symmetric}, and otherwise asymmetric, which more precisely typically refers to the case $n_r < n_t$. For the time being, a space--time code $\mathcal{X}$ will just be a finite collection of complex matrices in $\Mat(n_t\times T,  \C)$. The channel equation in this multiple-input multiple-output (MIMO) setting is given by
\begin{equation}
\label{eqn:mimo}
	Y_{n_r\times T} = H_{n_r\times n_t}X_{n_t\times T} + N_{n_r\times T}, 
\end{equation}
where $Y$ is the received matrix, and $X = \left[x_{ij}\right]_{i,j}\in\mathcal{X}$ is the \emph{space--time code} matrix. In the above equation, we adopt the Rayleigh fading channel model, \emph{i.e.}, the entries of the random \emph{channel matrix} $H = \left[h_{ij}\right]_{i,j}$ are complex variables with identically distributed real and imaginary parts,
\begin{align*}
	\Re(h_{ij}), \Im(h_{ij}) \sim \mathcal{N}(0,\sigma_h^2),
\end{align*}
yielding a Rayleigh distributed envelope 
\begin{align*}
	|h_{ij}| = \sqrt{\Re(h_{ij})^2+\Im(h_{ij})^2} \sim \text{Ray}(\sigma_h)
\end{align*} 
with scale parameter $\sigma_h$. We assume further that the channel remains static during the entire transmission of the codeword matrix $X$, and then changes independently of its previous state. The additive noise\footnote{The noise is a combination of thermal noise and noise caused by the signal impulse.} is modeled by the \emph{noise matrix} $N$, whose entries are independent, identically distributed complex Gaussian random variables with zero mean. 


Let us briefly discuss what constitutes a "good" code. Consider a space--time code $\mathcal{X}$, and let $X, X'$ be code matrices ranging over $\mathcal{X}$.
Two basic design criteria can be derived in order to minimize the probability of error \cite{TSC}. 
\begin{itemize}
	\item[i)] The \emph{diversity gain} of a code is the asymptotic slope of the error probability curve with respect to the signal-to-noise ratio ($\SNR$) in a $\log-\log$ scale, and relates to the minimum rank $\rk(X-X')$ over all pairs of distinct code matrices $(X,X') \in \mathcal{X}^2$. The minimum rank of $\mathcal{X}$ should satisfy 
	\begin{align*}
		\min_{X \neq X'} \rk(X-X') = \min\{n_t,T\},
	\end{align*} 
	in which case $\mathcal{X}$ is called a \emph{full-diversity} code.

	\item[ii)] 
	The \emph{coding gain} measures the difference in $\SNR$ required for two different codes to achieve the same error probability. For a full-diversity code this is proportional to the determinant
\begin{align*}
	\det \left( (X-X')(X-X')^{\dagger}\right).
\end{align*}
\end{itemize}

We define the \emph{minimum determinant} of a code $\mathcal{X}$ as the infimum 
\begin{align*}
	\Delta_{\min}(\mathcal{X}) := \inf\limits_{X \neq X'}{\det \left( (X-X')(X-X')^{\dagger}\right)}
\end{align*}
as the code size increases, $|\mathcal{X}| \to \infty$. If $\Delta_{\min}(\mathcal{X})  > 0$, the space--time code is said to have the \emph{nonvanishing determinant} property \cite{BR}.  In other words, a nonvanishing determinant guarantees that the minimum determinant is bounded from below by a positive constant even in the limit, and hence the error probability will not blow up when increasing the code size. 


In 2003, the usefulness of central simple algebras to construct space--time codes meeting both of the above criteria was established in \cite{SRS}; especially of (cyclic) division algebras, for which the property of being division immediately implies full diversity. Thereupon the construction of space--time codes started to rely on cleverly designed algebraic structures,  leading to the construction of multiple extraordinary codes, such as the celebrated Golden code \cite{BRV}, or general Perfect codes \cite{ORBV,ESK}.
It was later shown in \cite{BR} that in a cyclic division algebra based code, achieving the nonvanishing determinant property can be ensured by restricting the entries of the codewords to certain subrings of the algebra alongside with a smart choice for the base field, and that ensuring nonvanishing determinants is enough to achieve the optimal trade-off between diversity and multiplexing. 

Further investigation carried out in \cite{HLL,VHLR} showed that codes constructed from orders, in particular \emph{maximal orders}, of cyclic division algebras performed exceptionally well. The main observation is that the discriminant of the order is directly related to the offered coding gain, and should be as small as possible in order to maximize the coding gain. 


Maximal orders were then the obvious candidates, as they maximize the normalised density of the corresponding lattice and hence also maximize the coding gain. Unfortunately, they are in general  very difficult to compute and may result in highly skewed lattices making the bit labeling a delicate problem on its own. Therefore, \emph{natural orders} with a simpler structure have become a more frequent choice as they provide a good compromise between the two common extremes: using maximal orders to optimize coding gain, on the one hand, and restricting to orthogonal lattices to simplify bit labeling, encoding, and decoding, on the other.


However, the current explicit constructions are typically limited to the symmetric case, while the asymmetric case remains largely open. The main goal of this article is to fill this gap,
though our interest is not to analyze the performance of explicit codes. Instead, we focus on the algebraic setup and provide lower bounds for the smallest possible discriminants of natural orders for the considered setups, and give explicit field extensions and corresponding cyclic division algebras meeting the lower bounds. 

This article is structured as follows. In Section~\ref{sec:stc} we will shortly introduce MIMO space--time coding and the construction of space--time codes using representations of orders in central simple algebras. Section~\ref{sec:nat_orders} contains the main results of this article. We will consider the most interesting asymmetric MIMO channel setups and fix $F = \Q$ or $F = \Q(i)$ as the base field to guarantee the nonvanishing determinant property\footnote{From a mathematical point of view, any imaginary quadratic number field would give a nonvanishing determinant, but the choice $\Q(i)$ matches with the quadrature amplitude modulation (QAM) commonly used in engineering.}. For each considered setup $(F, n_t, n_r)$, we will find an explicit field extension $F \subset L \subset E$ and an explicit $L$-central cyclic division algebra over $E$, such that the norm of the discriminant of its natural order is minimal. This will translate into the largest possible determinant (see \eqref{eqn:density} and \cite{VHLR,HL} for the proof) and thus provide us with the maximal coding gain one can achieve by using a natural order.

\section{Space--time codes from orders in central simple algebras}
\label{sec:stc}
From now on, and for the sake of simplicity, we set the number $n_t$ of transmit antennas equal to the number $T$ of time slots used for transmission and shortly denote $n:=n_t=T$. Thus, the considered codewords  will be square matrices.

\subsection{Space--time lattice codes}
\label{sec:sec2}

Very simplistically defined, a space--time code is a finite set of complex matrices. However, in order to avoid accumulation points at the receiver, in practical implementations it is convenient to impose an additional discrete structure on the code, such as a lattice structure. We define a \emph{space--time code} to be a finite subset of a \emph{lattice} 
$$
\Lambda =\left\{\sum_{i=1}^k z_i B_i\,|\, z_i\in\Z\right\}\subset \Mat(n,\C),
$$ 
where $\{B_1,\ldots,B_k\}\subset \Mat(n,\C)$ is a \emph{lattice basis}. We recall that a lattice in $\Mat(n,\C)$ is \emph{full} if $\rank(\Lambda) := k= 2 n^2$. We call a space--time lattice code \emph{symmetric}, if its underlying lattice is full, and \emph{asymmetric}\footnote{This definition relates to the fact that a symmetric code carries the maximum amount of information (\emph{i.e.}, dimensions) that can be transmitted over a symmetric channel without causing accumulation points at the receiving end. In an asymmetric channel, a symmetric code will result in accumulation points, and hence asymmetric codes, \emph{i.e.}, non-full lattices are called for. See \cite{HL} for more details.} otherwise. 

Due to linearity, given a lattice $\Lambda \subset \Mat(n,\C)$ and $X, X' \in \Lambda$,
\begin{align*}
	\Delta_{\min}(\Lambda) := \inf\limits_{X \neq X'} \det \left((X-X')(X-X')^{\dagger} \right) =
 \inf\limits_{X \in \Lambda\backslash\left\{0\right\}}{\left|\det(X)\right|^2}.
\end{align*}
This implies that any lattice $\Lambda$ satisfying the nonvanishing determinant property can be scaled so that $\Delta_{\min}(\Lambda) $ achieves any wanted nonzero value. 
Consequently, a meaningful comparison of different lattices requires some kind of normalization. To this end, 
consider the Gram matrix of $\Lambda$,
\begin{align*}
	G_{\Lambda} := \left[\Re\left(\Tr\left(B_i B_j^\dagger\right)\right)\right]_{1 \le i,j \le k},
\end{align*}
where $\Tr$ denotes the matrix trace. The volume $\nu(\Lambda)$ of $\Lambda$ is related to the Gram matrix as $\nu(\Lambda)^2 = \det(G_{\Lambda})$. 

\begin{itemize}
	\item[i)] The \emph{normalized minimum determinant} \cite{VHLR} of $\Lambda$ is the  minimum determinant of $\Lambda$ after scaling it to have a unit size fundamental parallelotope, that is,
		\begin{align*}
			\delta(\Lambda) = \frac{\Delta_{\min}(\Lambda)}{\nu(\Lambda)^{\frac{n}{k}}}. 
		\end{align*}
	\item[ii)] The \emph{normalized density} \cite{VHLR} of $\Lambda$ is 
		\begin{align}
		\label{eqn:norm_density}
			\mu(\Lambda) = \frac{\Delta_{\min}(\Lambda)^{\frac{k}{n}}}{\nu(\Lambda)}. 
		\end{align}
\end{itemize}

We  get the immediate relation $\delta(\Lambda) = \mu(\Lambda)^{{\frac{n}{k}}}$, from which it follows that in order to maximize the coding gain it suffices to maximize the density of the lattice. Maximizing the density, for its part, translates into a certain \emph{discriminant minimization problem} \cite{VHLR,HL}, as we shall see in Section~\ref{subsec:stc_orders} (cf. \eqref{eqn:density}). This observation is crucial and will be the main motivation underlying Section~\ref{sec:nat_orders}.

\subsection{Central simple algebras and orders}

We recall that a finite dimensional algebra  over a number field $L$ is an $L$-\emph{central simple algebra}, if its center is precisely $L$ and it has no nontrivial ideals. 
An algebra is said to be \emph{division} if all of its nonzero elements have a multiplicative inverse.
By \cite[Prop.~1]{SRS}, as long as the underlying algebraic structure of a space--time code is a division algebra, the full-diversity property of the code  will be guaranteed.
It turns out that if  $L$ is an algebraic number field, then every $L$-central simple algebra is a \emph{cyclic algebra} \cite[Thm.~32.20]{Re}. 

Let $E/L$ be a cyclic extension of number fields of degree $n$ with respective rings of integers $\Or_E$ and $\Or_L$, and cyclic Galois group $\G(E/L) = \langle \sigma \rangle$. We fix a nonzero element $\gamma\in L^\times$ and consider the right $E$-vector space
\begin{align*}
	\mcc := (E/L,\sigma,\gamma) = \bigoplus_{i = 0}^{n-1}u^i E,
\end{align*}
with left multiplication defined by $xu= u\sigma(x)$ for all $ x\in E$, and $u^n=\gamma$. The triple $\mcc$ is referred to as a \emph{cyclic algebra} of \emph{index} $n$.

The obvious choice of lattices in $\mcc$ will be its \emph{orders}. We recall that if $R \subset L$ is a Dedekind ring, an $R$-order in $\mcc$ is a subring $\Or \subset \mcc$ which shares the same identity as $\mcc$, is a finitely generated $R$-module, and generates $\mcc$ as a linear space over $L$. Furthermore, an order is \emph{maximal} if it not properly contained in any other $R$-order of $\mcc$. Of special interest in this article is the $\mcl$-module
\begin{align*}
	\Or_{\nat} := \bigoplus\limits_{i=0}^{n-1}{u^i\mce},
\end{align*}
which we refer to as the \emph{natural order} of $\mcc$. 

Throughout the paper, we will denote the relative field norm map of the extension $E/L$ by $\Nm_{E/L}$ and the absolute norm map  by $\Nm_E=\Nm_{E/\Q}$. The restriction of this map to orders may be specified in the notation as $\Nm_{\Or_E/\Or_L}$ and $\Nm_{\Or_E}=\Nm_{\Or_E/\Z}$.

If the element $\gamma$ fails to be an algebraic integer, then $\Or_{\nat}$ will not be closed under multiplication. Furthermore, a necessary and sufficient condition for an index-$n$ cyclic algebra $(E/L,\sigma,\gamma)$ to be division is that
\begin{align}
\label{eqn:non_norm}
	\gamma^{n/p} \notin \Nm_{E/L}(E^\times)
\end{align} 
for all primes $p \mid n$. This is a simple extension of a well-known result due to A. Albert, for more details and a proof see  \cite[Prop. 3.6]{VHLR}. In what follows we will refer to such a non-zero element $\gamma \in \mcl$ as a \emph{non-norm element} for $E/L$.  

\begin{remark}
We recall that given a Dedekind ring $R \subset L$ and an $R$-order $\Or$ with basis $\{x_1,\ldots, x_n^2\}$ over $R$, the $R$-discriminant of $\Or$ is the ideal  
\begin{align*}
	\disc(\Or/R) = \left(\det\left(\tr_{\mcc/L}(x_i x_j)_{i,j=1}^{n^2}\right)\right),
\end{align*}
where $\tr(\cdot)$ denotes the reduced trace, which will be defined in \eqref{eqn:lrr}.

While the ring of algebraic integers is the unique maximal order in an algebraic number field, an $L$-central simple division algebra may contain several maximal orders. They all share the same discriminant \cite[Thm.~25.3]{Re}, known as the discriminant $d_{\mcc}$ of the  algebra ${\mcc}$. Given two $\mcl$-orders $\Gamma_1$, $\Gamma_2$, it is clear that if $\Gamma_1\subseteq\Gamma_2$, then $\disc(\Gamma_2/\mcl) \mid \disc(\Gamma_1/\mcl$). Consequently, $d_{\mcc}|\disc(\Gamma/\mcl)$ for every $\mcl$-order $\Gamma$ in ${\mcc}$, and the ideal norm $\Nm_{\mcl}(d_{\mcc})$ is the smallest possible among all  $\mcl$-orders of $\mcc$.
\end{remark}

\subsection{Algebraic space--time codes from representations of orders}
\label{subsec:stc_orders}

Let $\mcc = (E/L,\sigma,\gamma)$ be a cyclic division algebra of index $n$. We fix compatible embeddings of $L$ and $E$ into $\C$, and identify $L$ and $E$ with their images under these embeddings. 
The $E$-linear transformation of $\mcc$ given by left multiplication by an element $c \in \mcc$ results in an $L$-algebra homomorphism
$\rho: \mcc \to \Mat(n,E),$
to which we refer to as the \emph{maximal representation}. 
An element $c = c_0 +  u c_1+ \cdots + u^{n-1} c_{n-1} \in \mcc$ can be identified via $\rho$
with the matrix 
\begin{equation}
\label{eqn:lrr}
	\rho(c) = \begin{bmatrix}
 c_0& \gamma\sigma(c_{n-1})&\gamma \sigma^2(c_{n-2})& \ldots &\gamma \sigma^{n-1}(c_{1})\\
c_1& \sigma(c_{0})&\gamma\sigma^2(c_{n-1})& \ldots &\gamma \sigma^{n-1}(c_{2})\\
\vdots & \vdots & \vdots & & \vdots \\
c_{n-1} & \sigma(c_{n-2}) & \sigma^2(c_{n-3})& \cdots &\sigma^{n-1}(c_{0})
\end{bmatrix}.
\end{equation}
The determinant $\nr_{\mathcal{C}/L}(c) := \det(\rho(c))$ and trace $\tr_{\mathcal{C}/L}(c) := \Tr(\rho(c))$ define the \emph{reduced norm} and \emph{reduced trace} of $c \in \mathcal{C}$, respectively. We may shortly denote $\nr=\nr_{\mathcal{C}/L}$ and $\tr=\tr_{\mathcal{C}/L}$, when there is no danger of confusion.

Next, given an order $\Or$ in $\mcc$, we may use the maximal representation to define an injective map $\rho: \Or \hookrightarrow \Mat(n,E) \subset \Mat(n, \C)$. If the center $L$ of the algebra is quadratic imaginary and $\Or$ admits an $\Or_L$ basis, then $\rho(\Or)$ is a lattice, and any finite subset $\mathcal{X}$ of $\rho(\Or)$ will be a space--time lattice code, in the literature often referred to as \emph{algebraic space--time code}. 	

\begin{remark}
Due to the algebra being division, the above matrices will be invertible, and hence any algebraic space--time code constructed in this way will have full diversity. Moreover, if $c \in \Or\backslash \left\{0\right\}$, we have $\nr(c) \in \Or_L\backslash\left\{0\right\}$ \cite[Thm. 10.1]{Re}, guaranteeing nonvanishing determinants for $L = \Q$ or quadratic imaginary.  
\end{remark}

We now relate the minimum determinant of a code to the density of the underlying lattice $\rho(\Or)$, which is the main motivation for choosing orders with small discriminant. If the center $L$ of the cyclic algebra is quadratic imaginary and the considered order $\Or$ admits an $\Or_L$-basis, the volume $\nu(\rho(\Or))$ of the lattice relates to the discriminant of the order as \cite{VHLR}
\begin{align}
\label{eqn:density}
	\nu(\rho(\Or)) = c(L,n)\left|\disc(\Or/\Or_L)\right|,
\end{align}
where $c(L,n)$ is a constant which depends on the center and $[E:L]$. Thus, for fixed minimum determinant, the density of the code (cf. \eqref{eqn:norm_density}) -- and consequently the coding gain -- is maximized by minimizing the discriminant of the order.

\begin{example}
\label{exp:golden}
Let $E/L$ be a quadratic real extension of number fields with Galois group $\G(E/L) = \langle \sigma \rangle$. Let $L$ be of class number 1 and  $\mcl$, $\mce = \mcl[\omega]$ the respective rings of integers.
We choose $\gamma \in \Or_L\setminus\left\{0\right\}$ such that $\gamma \notin \Nm_{E/L}(E^\times)$, and define the cyclic division algebra
\begin{align*}
	\mathcal{C} = (E/L,\sigma,\gamma) = E \oplus uE,
\end{align*} 
where $u^2 = \gamma$. We consider the natural order $\Or_\text{nat}$ of $\mcc$ and construct an algebraic space--time code as a finite subset 
\begin{align*}
\mathcal{X} \subset \left\{
\begin{bmatrix} x_1 + x_2\omega  & \gamma(x_3 + x_4\sigma(\omega)) \\ x_3 + x_4\omega & x_1 + x_2\sigma(\omega) \end{bmatrix} \ \middle|\ x_i\in \mcl
\right\}= \rho(\Or_\text{nat}).
\end{align*}
 
The choice of fields $(E,L) = (\Q(i,\sqrt{5}),\Q(i))$ and element $\gamma = i$ gives rise to the Golden algebra and to the well-known \emph{Golden code} \cite{ORBV}. 
\end{example}

The above example relates to the symmetric scenario, \emph{i.e.}, it has an underlying lattice that is full. Full lattices can be efficiently decoded when $n_r = n_t$, and the same lattice codes can also be employed when $n_r > n_t$.  There is no simple optimal decoding method for symmetric codes, however, when $n_r < n_t$\footnote{Having too few receive antennas will cause the lattice to collapse resulting in accumulation points, since the received signal now has dimension $2n_rn_t<2n_t^2$. Hence, partial brute-force decoding of high complexity has to be carried out.}. 

Building upon symmetric codes, we now briefly introduce \emph{block diagonal asymmetric} space--time codes \cite{HL}, better suited for the asymmetric scenario. 
 Let $F \subset L \subset E$ be a tower of field extensions with extension degrees $[E:L] = n_r$, $[L:F] = n$, and $[E:F]=n_t=n_rn$, and with Galois groups $\G(E/F) = \langle \tau\rangle $ and $\G(E/L) = \langle \sigma\rangle= \langle \tau^n \rangle$. We fix a non-norm element $\gamma \in \mcl\setminus\{0\}$, and consider the cyclic division algebra
\begin{align*}
	\mcc = (E/L,\sigma,\gamma) = \bigoplus\limits_{i=0}^{n_r-1}u^i E.
\end{align*}
Given any order  $\Or $  in $ \mcc$, we identify each element $c \in \Or$ with its maximal representation $\rho(c)$  and construct the following infinite block-diagonal lattice achieving the nonvanishing determinant property, provided that the base field $F$ is either $\mathbb{Q}$ or quadratic imaginary \cite{HL}:
 
\begin{align*}
	\mathcal{L}(\Or) = \left\{\left.\begin{bmatrix} \rho(c) & 0 & \cdots & 0 \\ 0 & \tau\left(\rho(c)\right) & & 0 \\ \vdots & & \ddots & \vdots \\ 0 & \cdots & 0 & \tau^{n-1}\left(\rho(c)\right) \end{bmatrix} \in \Mat(n_t,\C)\, \right|\, c \in \Or \right\}.
\end{align*} 

\begin{remark}
The \emph{code rate} \cite{HL} of a space--time code carved out from $\mathcal{L}(\Or)$ in (complex) symbols per channel use is 
\begin{align*}
  R = \begin{cases}
  		nn_r^2/nn_r=n_r &\mbox{if $F$ is quadratic imaginary,} \\ 
  		nn_r^2/2nn_r=n_r/2 &\mbox{if $F=\mathbb{Q}$.} 
	\end{cases} 
\end{align*}
We point out that $n_r$ is the maximum code rate that allows for avoiding accumulation points at the receiving end with $n_r$ receive antennas.
\end{remark}

In summary, in order to construct an algebraic space--time code, we first choose a central simple algebra over a suitable base field and then look for a dense lattice in it. This amounts to selecting an adequate order in the algebra. As motivated earlier, we will opt for natural orders as a compromise between simplicity and maximal coding gain. 

\section{Natural orders with minimal discriminant}
\label{sec:nat_orders}

As an illustration of the general algebraic setup, consider the tower of extensions depicted in Figure~\ref{fig:tower}. 
\begin{figure}[!h]
\centering
	\includegraphics[trim={4cm 18.5cm 13cm 4.5cm},clip,scale=.9]{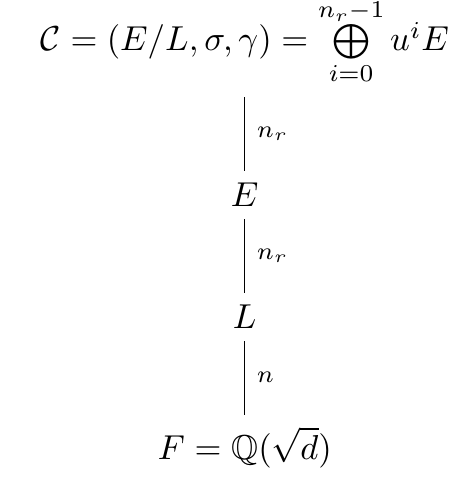}
\caption{Tower of Field Extensions.}
\label{fig:tower}
\end{figure}
In order to get the nonvanishing determinant property, we fix the base field $F \in \left\{\mathbb{Q},\mathbb{Q}(i)\right\}$, as well as the extension degrees $n = [L:F]$ and $n_r = [E:L]$. With these parameters fixed, our goal is to find an explicit field extension $E/L$ with $\G(E/L) = \langle \sigma \rangle$ and a non-norm element $\gamma \in \mcl\setminus\{0\}$ such that $E/F$ is a cyclic extension, 
$(E/L,\sigma,\gamma)$ is a cyclic division algebra,
and the absolute value $|\Nm_{\Or_F}(\disc(\Or_{\nat}/\mcf))|$ is the minimum possible among all cyclic division algebras satisfying the fixed conditions. 
Our constructions rely on some key properties of cyclic division algebras and their orders that we will next present as lemmata. 


\begin{lemma}{\cite[Lem.~5.4]{VHLR} and \cite[Prop.~5.3]{HL}}
\label{pro:disc}
Let $(E/L,\sigma,\gamma)$ be a cyclic division algebra of index $n_r$ and $\gamma \in \mcl\setminus\{0\}$ a non-norm element. We have  
\begin{align*}
	\disc(\Or_{\nat}/\mcl) = \disc(E/L)^{n_r}\cdot\gamma^{n_r(n_r-1)},
\end{align*} 
with $\disc(E/L)$ the $\Or_L$-discriminant of $\Or_E$. Hence, if $F \subset L$, then by the discriminant tower formula
\begin{align}
\label{eqn:disc}
\begin{split}
	\disc(\Or_{\nat}/\mcf) &=
\Nm_{L/F}(\disc(\Or_{\nat}/\mcl))\cdot \disc(L/F)^{n_r^2} \\
&= \disc(E/F)^{n_r}\cdot \Nm_{L/F}^{n_r(n_r-1)}(\gamma).
\end{split}
\end{align}
\end{lemma}

\begin{lemma} \cite[Thm.~2.4.26]{Veh}
\label{pro:bound}
Let $L$ be a number field and $(\mathfrak{p}_1, \mathfrak{p}_2)$ a pair of norm-wise smallest prime ideals in $\mcl$. If we do not allow ramification on infinite primes, then the smallest possible discriminant of all central division algebras over $L$ of index $n_r$ is 
\begin{equation}
\label{eqn:bound}
	(\mathfrak{p}_1\mathfrak{p}_2)^{n_r(n_r-1)}.
\end{equation}
\end{lemma}

We have arrived at the following \emph{optimization problem:} in order to minimize the discriminant $\disc(\Or_{\nat}/\mcf)$ of a natural order of an index-$n_r$ L-central division algebra over a fixed base field $F\subset L$, we must jointly minimize the relative discriminant of the extension $E/F$ and the relative norm of the non-norm element $\gamma$. 

Our findings are summarized in the following table and will be proved, row by row, in the subsequent five theorems. Here, $\alpha$ is a root of the polynomial $X^2+X-i$ and $\beta$ denotes a root of the polynomial $X^3-(1+i)X^2+5iX-(1+4i)$.
 
\begin{table}[H]
\scalebox{0.94}{
\begin{tabular}{lllllllll}
	$F$ & $n$ & $n_r$ & $n_t$ & rate & $\Nm_{\mcf}(\disc(\Or_{\nat}/\mcf))$ & $L$ & $E$ & $\gamma$ \\
	\hline
 	$\Q$ & 1 & 2 & 2 & 1 & $2^2\cdot 3^2$ & $\Q$ & $\Q(i\sqrt{3})$ & 2 \\
 	$\Q$ & 2 & 2 & 4 & 1 & $2^4\cdot 5^6$ & $\Q(\sqrt{5})$ & $\Q(\zeta_5)$ & -4 \\
 	$\Q(i)$ & 2 & 2 & 4 & 2 & $2^4\cdot 17^3 $ & $\Q(i,\alpha)$ & $\Q(i, \sqrt{\alpha} )$ &$1+ i$ \\
	$\Q(i)$ & 2 & 3 & 6 & 3 & $3^{18}\cdot 13^{12}$ & $\Q(i, i\sqrt{3})$ & $\Q(i, i\sqrt{3},\beta)$ & $\frac{1+i\sqrt{3}}{2}$ \\
	$\Q(i)$ & 3 & 2 & 6 & 2 & $2^6\cdot 3^{12}\cdot 13^8$ & $\Q(i, \beta)$ & $\Q(i, i\sqrt{3},\beta)$ & $1+i$
\end{tabular}}  
\caption{Main results summarized.}
\label{tab:results}
\end{table}

\begin{remark} It is often preferred that $|\gamma|=1$ for balanced transmission power. However, there are good `remedy' techniques for the case when $|\gamma|>1$, see \emph{e.g.}, \cite{ESK,VHO}. 
\end{remark}

\subsection{General Strategy}

We briefly elaborate on the three-step strategy that we will follow to prove each of the theorems. Let $F$, $n$ and $n_r$ be fixed.

\paragraph{\underline{Step 1.}}
We start by finding an explicit cyclic extension $E/L$ of degree $n_r$, $F\subset L$, such that $[L:F] = n$ and $E/F$ is cyclic with $|\Nm_{\Or_F}(\disc(E/F))|$ smallest possible. In the cases where $F = \mathbb{Q}$, our extension $E/F$ will either be quadratic or quartic cyclic, and we simply use well-known formulas for computing $\disc(E/\Q)$. 
For $F = \Q(i)$ we resort to the following results from Class Field Theory (all the details can be found in \cite{CS}).  

Let $F\subset E $ be an abelian extension of number fields. For each prime $\mathfrak p$ of $F$ that is unramified in $E$ there is a unique element $\frob_{\mathfrak p}\in \G(E/F)$ that induces the Frobenius automorphism $x\mapsto x^{\sharp k_{\mathfrak p}}$ on the residue field extensions $k_{\mathfrak p}\subset k_{\mathfrak q}$ for the primes $\mathfrak q $ in $E$ extending $\mathfrak p$. The order of $\frob_{\mathfrak p}$ in $\G(E/F)$ equals the residue class degree $[k_{\mathfrak q}:k_{\mathfrak p}]$, and the subgroup $\langle\frob_{\mathfrak p}\rangle$ of $\G(E/F)$ is the decomposition group of $\mathfrak p$. The \textit{Artin map} for $E/F$ is the homomorphism 
\begin{align*}
	\psi_{E/F}: I_F(\disc(E/F))&\longrightarrow \G(E/F)\\
	\mathfrak p &\longmapsto \frob_{\mathfrak p}
\end{align*}
on the group $I_F(\disc(E/F))$ of fractional $\Or_F$- ideals generated by the prime ideals $\mathfrak p$ of $F$ that do not divide the discriminant $\disc(E/F)$. These are unramified in $E$. For an ideal $\mathfrak a$ in $I_{F}(\disc(E/F))$ we call $\psi_{E/F}(\mathfrak a)$ the \textit{Artin symbol} of $\mathfrak a$ in $\G(E/F)$.

The \textit{Artin Reciprocity Law} states that if $F\subset E$ is an abelian extension, then there exists a nonzero ideal $\mathfrak m_0\Or_L$ such that the kernel of the Artin map $\psi_{E/L}$ contains all principal $\Or_L$-ideals $x\Or_F$ with $x$ totally positive and $x\equiv 1\zmod {\mathfrak m_0}$. 
We define a \emph{modulus} of $L$ to be a formal product $\mathfrak m = \mathfrak m_0 \mathfrak m_\infty$, where $\mathfrak m_0$ is a nonzero $\Or_E$-ideal and $\mathfrak m_\infty$ is a subset of the real primes of $L$. We write $x\equiv 1 \zmod{^{\times}\mathfrak m}$ if $\hbox{ord}_{\mathfrak p}(x-1)\ge\hbox{ord}_{\mathfrak p}(\mathfrak m_0)$ at the primes $\mathfrak p$ dividing the finite part $\mathfrak m_0$ and $x$ is positive at the real primes in the infinite part $\mathfrak m_{\infty}$. 
In the language of moduli, Artin's Reciprocity Law asserts that there exists a modulus $\mathfrak m$ such that the kernel of the Artin map contains the \textit{ray group} $R_{\mathfrak m}$ of principal $\Or_L$-ideals $x\Or_L$ generated by elements $x\equiv 1 \zmod{^{\times}\mathfrak m}$. The set of these \textit{admissible} moduli for $L/E$ consists of the multiples of some minimal modulus $\mathfrak f_{E/L}$, the \textit{conductor} of $L/E$. The primes occurring in $\mathfrak f_{E/L}$ are the primes of $L$, both finite and infinite, that ramify in $E$. 
 
 If $\mathfrak m=\mathfrak m_0\mathfrak m_{\infty}$ is an admissible modulus for $L/E$, and $I_{\mathfrak m}$ denotes the group of fractional $\Or_L$-ideals generated by the primes $\mathfrak p$ coprime to $\mathfrak m_0$, then the Artin map induces a surjective homomorphism 
\begin{align}
\label{eqn:artin}
\begin{split}
	\psi_{E/L}: \Cl_{\mathfrak m} = I_{\mathfrak m}/R_{\mathfrak m} &\longrightarrow \G(E/L),\\
	[\mathfrak p] &\longmapsto \frob_{\mathfrak p},
\end{split} 
\end{align}
where $\Cl_{\mathfrak m}$ is the \emph{ray class group}, and $\ker(\psi_{E/F}) = A_{\mathfrak m}/R_{\mathfrak m}$ with
\begin{equation} 
\label{eqn:norm}
	A_{\mathfrak m} = \Nm_{E/F}(I_{\mathfrak m\Or_F})\cdot R_{\mathfrak m}.
\end{equation}

The \textit{existence theorem} from Class Field Theory states that for every modulus $\mathfrak m$ of $F$, there exists an extension $F\subset H_{\mathfrak m}$ for which the map in \eqref{eqn:artin} is an isomorphism. Inside some fixed algebraic closure of $F$, the \textit{ray class field} $H_{\mathfrak m}$ is uniquely determined as the maximal abelian extension  of $F$ in which all the primes in the ray group $R_{\mathfrak m}$ split completely. Conversely, if $F\subset E$ is abelian then $E \subset H_{\mathfrak m}$ whenever $\mathfrak m$ is an admissible modulus for $F\subset E$. For $E = H_{\mathfrak m}$, we have $A_{\mathfrak m} = R_{\mathfrak m}$ in \eqref{eqn:norm} and an \textit{Artin isomorphism} 
\begin{equation}
\label{eqn:artiniso}
	\Cl_{\mathfrak m} \simeq \G(H_{\mathfrak m}/F).
\end{equation}
	
For all $\mathfrak m$, the ray group $R_{\mathfrak m}$ is contained in the subgroup $P_{\mathfrak m} \subset I_{\mathfrak m}$ of principal ideals in $I_{\mathfrak m}$, with $I_{\mathfrak m}/P_{\mathfrak m} = \Cl_{F}$, the class group of $F$. There is a natural exact sequence
\begin{equation}
\label{eqn:exact}
	\Or_{F}^{\times} \longrightarrow (\Or_F/\mathfrak m)^{\times} \longrightarrow \Cl_{\mathfrak m} \longrightarrow \Cl_{F}\longrightarrow 0,
\end{equation}
and the residue class in $\Cl_{\mathfrak m}$ of $x \in \Or_L$ coprime to $\mathfrak m_0$ in the finite group $(\Or_F/\mathfrak m)^{\times} = (\Or_F/\mathfrak m)^{\times}\times\prod_{\mathfrak p\mid \mathfrak m_{\infty}}\langle -1 \rangle $ consists of its ordinary residue class modulo $\mathfrak m_0$ and the signs of its images under the real primes $\mathfrak p \mid \mathfrak m_{\infty}$.
	
Finally, we compute the discriminant $\disc(E/F)$ using Hasse's conductor-discriminant formula
\begin{equation}
\label{eqn:condisc}
	\disc(E/F)=\prod_{\chi:I_{\mathfrak m}/A_{\mathfrak m}\to\C^{\times}}\mathfrak f(\chi)_0,
\end{equation}
where $\chi$ ranges over the characters of the finite group $I_{\mathfrak m}/A_{\mathfrak m}\simeq \G(E/F)$ and $\mathfrak f(\chi)_0$ denotes the finite part of the conductor $\mathfrak f(\chi)$ of the ideal group $A_{\chi}$ modulo $\mathfrak m$ satisfying $A_{\chi}/A_{\mathfrak m}=\ker \chi$.

\paragraph{\underline{Step 2.}}
Unfortunately, having $|\Nm_{\Or_F}(\disc(E/F))|$ smallest possible is not sufficient for $|\Nm_{\Or_F}(\disc(\Or_{\nat}/\mcf))|$ to be smallest possible as well. 
Using \eqref{eqn:disc} and \eqref{eqn:bound} we can derive a positive lower bound on the size of a non-norm element $\gamma \in \Or_L$ as  
\begin{equation}
\label{eqn:constant}
	\left|\Nm_{\Or_L}(\gamma ^{n_r-1})\right|\ge \frac{\left|\Nm_{\Or_L}(\mathfrak{p}_1\mathfrak{p}_2)^{n_r-1}\right|}{\left|\Nm_{\Or_L}(\disc(E/L))\right|} =: \lambda_{E,L} \in \mathbb{N},
\end{equation}
where $(\mathfrak{p}_1, \mathfrak{p}_2)$ is a pair of norm-wise smallest prime ideals in $\mcl$. 
If $\disc(E/L)$ is minimal and $\lambda_{E,L} > 1$, as will be the case in Theorems~\ref{thm:res1} and \ref{thm:res2} below, we need to balance the size of $\disc(E/F)$ and $\gamma$ in order to achieve minimality of $|\Nm_{\Or_F}(\disc(\Or_{\nat}/\mcf))|$. 

In the last three theorems, we will then proceed as follows. Given two cyclic algebras $(E/L, \sigma, \gamma)$ and $(E'/L',\sigma', \gamma')$ of index $n_r$, where $\langle \sigma \rangle = \G(E/L)$, $\langle \sigma' \rangle = \G(E'/L')$ and such that $\Q(i)\subset L, L'$, we have by \eqref{eqn:disc}, since norms in $\Q(i)$ are positive, 
\begin{align}
\label{eqn:discsize}
\begin{split}
	\Nm_{\Z[i]}(\disc(\Or_{\nat}/\Z[i])) &\le \Nm_{\Z[i]}(\disc(\Or'_{\nat}/\Z[i])) \\
	&\Leftrightarrow \\
	D_{E/L}(\gamma) &\le D_{E'/L'}(\gamma'),  
\end{split}
\end{align}
with  
\begin{equation}
\label{eqn:balance}
	D_{E/L}(\gamma) := \Nm_{\Z[i]}(\disc(E/\Q(i)))\cdot \Nm_{\Or_L}(\gamma)^{n_r-1}.
\end{equation} 

Our strategy will be to fix a non-norm element $\gamma \in \Or_L$ of smallest possible norm, compute  $D_{E/L}(\gamma)$ and, along the lines of Step 1, compare $D_{E/L}(\gamma)$ with $D_{E'/L'}(\gamma')$, where $E'/L'$ runs over all degree-$n_r$ cyclic field extensions such that $\Q(i)\subset L'$ and $\gamma' \in \Or_L$ is a non-norm element for $E'/L'$ of smallest possible norm.

\paragraph{\underline{Step 3.}}
If $\disc (E/L)$ is smallest possible and $\lambda_{E,L} < 1$, as will be the case in Theorems \ref{thm:res3}, \ref{thm:res4} and \ref{thm:res5}, the optimal situation would be to be able to choose a non-norm element $\gamma$ which is a unit, $\gamma \in \Or_L^{\times}$. Hasse's Norm Theorem
will help us decide whether such an element exists and, if it does, how to find it. We use the following strategy from the theory of local fields to compute $\Nm_{K/k}(K^{\times})$ when $K/k$ is an extension of non-Archimedean local fields (all the details can be found in \cite[Chp.~7]{JM}).

Let $k$ be a field, complete under a discrete valuation $v_k$. Let $A_k$ be its valuation ring with maximal ideal $\mathfrak m_k$, generated by a local uniformizing parameter $\pi_k$ with $v_k(\pi_k)=1$, and $\overline{k} = A_k/\mathfrak m_k$ the residue field, where $|\overline k| = q_k$ is a power of a rational prime $p$. Denote by $U_k = A_k-\mathfrak m_k$ the multiplicative group of invertible elements $A_k^{\times}$ of $A_k$, and set $U_k^i = 1+\mathfrak m_k^i$, $i\ge 1$.
Then, $U_k = S_k\times U_k^1$, where $S_{k}$ is a complete set of representatives of $\overline k$, and  
\begin{align*}
	k^{\times} = \langle \pi_k\rangle U_k = S_k \times \langle \pi_k \rangle \times U_k^1.
\end{align*}  
As $\overline k^{\times}$ is cyclic of order $q_k-1$, we may take $S_k =\left\{0\right\} \cup \left\{\zeta_{q_k-1}^i \mid 1 \le i \le q_k-1\right\} $, where $ \zeta_n$ denotes a primitive $n$-th root of unity in $\overline k$.
 
Let $K$ be finite separable extension of $k$, $A_K$ the integral closure of $A_k$ in $K$, and let $v_K$, $\mathfrak m_K$, $\pi_K$, $\overline K$, $q_K$ and $U_K^i$ be defined as above. As usual, we denote the ramification index and residue degree of $\mathfrak m_K$ in $K/k$ by $e_{K/k}$ and $f_{K/k}$, respectively. We have $e_{K/k}\cdot f_{K/k} = [K:k]$ and $\pi_k = \pi_K^{e_{K/k}}\times u$ with $u$ a unit, so that $\Nm_{K/k}(\pi_K) = \pi_k^{f_{K/k}}$.

The group $\Nm_{K/k} (U_K)$ is a subgroup of $U_k$ with $[U_k: \Nm_{K/k}(U_K)]=e_{K/k} $, and $\Nm_{K/k}(U_K^1) \subset U_k^1$ with $[U_k^1: \Nm_{K/k}(U_K^1)]$ a power of $p$.
Consequently, if the extension is unramified, i.e. $e_{K/k}=1$, then  $\Nm_{K/k}(U_K)=U_k$ and every unit is a norm. 

Suppose hereinafter that $K/k$ is a totally tamely ramified extension, thus $f_{K/k} = 1$,  $(p, e_{K/k}) = 1$, $\overline K = \overline k$, $q_K = q_k$ and the \emph{local conductor}, \emph{i.e.}, the smallest integer $f$ such that $U_ k^f \subset \Nm_{K/k}(K^{\times})$, is 1 (\cite[Aside 1.9]{JM2}). Then $\Nm_{K/k}(\zeta_{q_K-1})=\zeta_{q_k-1}^{[K:k]}$ and $\Nm_{K/k}(U_K^1)=U_k^1$. 
Consequently, 
\begin{align*}
	\Nm_{K/k}(K^{\times}) = \langle \Nm_{K/k}(\pi_K), \zeta_{q_k-1}^{[K:k]}\rangle U_k^1.
\end{align*}
Since $\pi_K$ is not a unit, we conclude that in the totally tamely ramified case,
\begin{equation}
\label{eqn:nonnorm}
	U_k \cap \Nm_{K/k}(K^{\times})= \langle \zeta_{q_k-1}^{[K:k]}\rangle U_k^1.
\end{equation}
	
With the above information at hand, we go back to the extension $E/L$ found in Step 1 with $\lambda_{E,L} < 1$. In order to produce a suitable unit $\gamma \in \Or_L^{\times}$ which is not a local norm at some prime ramifying in the extension, we compute the ramified primes as well as $\langle \zeta_{q_k-1}^{[K:k]}\rangle U_k^1$ in the corresponding local extension $K/k$. Considering $\Or_L$ as a subset of $A_k$ and using Hensel's lemma, we look for a unit in $\Or_L^{\times}$ such that its image in $\overline{k}$ lies  outside $\langle \zeta_{q_k-1}^{[K:k]}\rangle$.
Unfortunately, if $\Or_L^{\times} \subset \langle \zeta_{q_k-1}^{[K:k]}\rangle U_k^1$, as will be the case in Theorems~\ref{thm:res3} and \ref{thm:res5}, a non-norm unit for $E/L$ will not exist. In those cases the sizes of $\disc(E/F)$ and $\gamma$ in \eqref{eqn:disc} must be balanced using \eqref{eqn:balance} in order to achieve minimality of $|\Nm_{\Or_F}(\disc(\Or_{\nat}/\mcf))|$.  

\subsection{Main Results}    

We are now ready to state and prove the main results of this article. 
    
\begin{theorem} 
\label{thm:res1}
Let $ \Q \subset E = \Q(\sqrt{d})$, $d \in \Z$ square-free, and $\G(E/L)=\langle \sigma \rangle$. Any cyclic division algebra $\left(E/\Q, \sigma, \gamma\right)$ of index 2 satisfies 
$|\disc(\Or_{\nat}/\Z)| \ge 36$, and equality is achieved for $E = \Q(i\sqrt{3})$, $\gamma = 2$.
\end{theorem}

\begin{proof}: The proof follows the strategy described above. 
\medskip

\noindent\emph{\underline{Step 1.}} 
The smallest possible quadratic discriminant over $\Q$ is $\disc(E/\mathbb{Q}) = 3$, corresponding to the field $E = \Q(i\sqrt{3}) = \Q(\omega)$,  $\omega$ a primitive cubic root of unity. Let $\G(E/\Q) = \langle \sigma \rangle$. 
\medskip

\noindent\emph{\underline{Step 2.}}  
A pair of smallest primes in $\Z$ is $(2,3)$, so that $\lambda_{E, \Q} = \frac{6}{3} > 1$  (cf. \eqref{eqn:constant}). Thus, any non-norm element $\gamma \in \Z$ satisfies $|\gamma | \ge 2$. 

To ensure that we can choose $\gamma = 2$, we first show that the equation $x^2+3y^2=2$ has no solution in $\Q$. Consequently, $2\not \in \Nm_{E/\Q}(E^{\times})$ and $(\Q(i\sqrt{3})/\Q, \sigma, 2)$ is a division algebra with  $\disc(\Or_{\nat}/\Z) = 36$.

Suppose that $\left(\frac{a}{b}\right)^2+3 \left(\frac{c}{d}\right )^2 = 2$, with  $a, b, c, d\in \Z$ and such that $(a,b) = (c,d) = 1$. Then,
\begin{equation}
\label{eqn:norm2}
	(ad)^2+3(bc)^2=2(bd)^2. 
\end{equation}
It is easy to deduce from \eqref{eqn:norm2} that $3^s, s \ge 0$, is the largest power of 3 dividing $b$ if and only if it is the largest power of $3$ dividing $d$. If $(3,bd) = 1$, then equation \eqref{eqn:norm2} has no solution in $\Z$, as 2 is not a square $\mathrm{mod}$ 3. Set $b = 3^sb'$ and $d = 3^sd'$, with $(3,b'd') = 1$, $s\ge 1$. Substituting into \eqref{eqn:norm2} yields
\begin{align*}
	(ad')^2+3(b'c)^2=2.3^s(b'd')^2,
\end{align*}
which is absurd, since $(3,ad') = 1$.

Next, we use \eqref{eqn:disc} to see that for $E' = \Q(\sqrt d)$ with $d \ne 1, -3$ a square-free integer and $\G(E'/\Q) = \langle \sigma '\rangle$, the cyclic division algebra $\left(E'/\Q, \sigma', \gamma'\right)$ satisfies $ |\disc(\Or_{nat}'/\Z)|>36$ for any choice of non-norm element $\gamma\in \Z$.
\begin{itemize}
	\item[i)] If $d\equiv 2,3 \zmod 4$, then $|\disc(E'/\Q)|=4|d| \ge 8$,
	and \eqref{eqn:disc} guarantees that  for any  $\gamma' \in \Z$, $|\disc (\Or_{nat}'/\Z)|\ge 8^2\gamma'^2\ge 64$.

	\item[ii)] If $d\equiv 1 \zmod 4$ and $|d|\ge 7$, then $\left|\disc(E'/\Q)\right| = |d|$ and \eqref{eqn:disc} implies $|\disc (\Or_{nat}'/\Z)|\ge 7^2\gamma'^2\ge 49$. 

	\item[iii)] For $d = 5$,  we have $\disc(E'/\Q) = 5$ and $\lambda_{E', \Q} = \frac{6}{5} = 2$ (cf. \eqref{eqn:constant}). Using \eqref{eqn:disc} we conclude that for any non-norm element $\gamma' \in \Z$, $\disc(\Or_{nat}'/\Z)\ge 5^2\cdot 2^2>36$.
\end{itemize} \qed
\end{proof}

\begin{theorem}
\label{thm:res2}
Let $\Q \subset L \subset E$ with $[E:\Q] = 4$, $[E:L] = 2$ and $\G(E/L)=\langle \sigma \rangle$. If $\left( E/L, \sigma, \gamma\right)$ is a cyclic division algebra, then $\Nm_\Z(\disc(\Or_{\nat}/\Z)) \ge 2^4\cdot 5^6$. Equality is achieved for $L = \Q(\sqrt{5}),$ $E = \Q(\zeta_5)$ and $\gamma = -4$, with $\zeta_5$ a primitive $5^\text{th}$ root of unity. 

\end{theorem}
  
\begin{proof}:
The fields $L$ and $E$ can be uniquely expressed as $L = \Q(\sqrt{D})$ and $E = \Q\left(\sqrt{A(D+B\sqrt D)}\right )$, with $A,B,C,D \in \Z$ such that $A$ is square-free and odd, $D = B^2+C^2$ is square-free, $B,C > 0$ and $(A,D) = 1$ (\cite{HW}, \cite{HSW}).
\medskip

\noindent\emph{\underline{Step 1.}}  

\begin{itemize}

	\item[i)]If $D \equiv 0 \zmod 2$, then $\disc(E/\Q) = 2^8\cdot A^2\cdot D^3 \ge 2^{11}$. The lower bound is attained for $D = 2$, $B = C = 1$, and $|A| = 1$. Using \eqref{eqn:disc} we deduce that $|\disc(\Or_{\nat}/\Z)| \ge 2^{22} > 2^4\cdot 5^6$ for any choice of $\gamma \in \Or_L$.

	\item[ii)] If $D \equiv B\equiv 1 \zmod 2$, then $\disc(E/\Q) = 2^6\cdot A^2 \cdot D^3 \ge 2^6\cdot 5^3$. This expression attains its minimum value for $D = 5$ and $|A| = B = 1$.  Hence, $|\disc(\Or_{\nat}/\Z)| \ge 2^{12}\cdot 5^6 > 2^4\cdot 5^6$ for any choice of $\gamma \in \Or_L$.

	\item[iii)] If $D\equiv 1 \zmod 2$, $B\equiv 0 \zmod 2$ and $A+B\equiv 3 \zmod 4$, we have $\disc(E/\Q)=2^4 \cdot A^2 \cdot D^3 \ge 2^4\cdot 5^3$. 
The minimum value is attained for $D = 5$, $B = 2$, and $A = 1$, so that
$|\disc(\Or_{\nat}/\Z)| \ge 2^8\cdot 5^6> 2^4\cdot 5^6$ for any choice of $\gamma \in \Or_L$.

	\item[iv)]  Finally, if $D\equiv 1 \zmod 2$, $B\equiv 0 \zmod 2$, $A+B\equiv 1 \zmod 4$ and $A\equiv \pm C \zmod 4$, we have $\disc(E/\Q) = A^2 \cdot D^3 \ge 5^3$. The minimum of this expression is attained for $D = 5$, $B = 2$ and $A = -1$, corresponding to the fields $E = \Q(\zeta_5)$ and $L = \Q(\zeta_5+\zeta_5^{-1}) = \Q(\sqrt 5)$. 
\end{itemize}

The last case provides us with a field extension $E/L = \Q(\zeta_5)/\Q(\sqrt 5)$ with $\disc(E/L) = 5$ and $\disc(E/\Q) = 5^3$. If we show $|\disc(\Or_{\nat}/\Z)| = 2^4\cdot 5^6$ for $\gamma = -4$ a non-norm element, the theorem will be proved.
\medskip

\noindent\emph{\underline{Step 2.}}  
The two smallest prime ideals in $\Or_L$ are $\mathfrak p_2 = 2\Or_L$ and $\mathfrak p_5$, a factor of $5\Or_L = \mathfrak p_5 \mathfrak p_5'$, of respective norms 4 and 5. Hence, $\lambda_{E,L} = 4$, and any suitable non-norm element for $E/L$ will satisfy $|\Nm_{\Or_L}(\gamma)|\ge 4$ (cf. \eqref{eqn:constant}). We show that $-4 \not \in \Nm_{E/L}(E^\times)$. Since the norm is multiplicative and $4 = \Nm_{E/L}(2)$, it suffices to show that $-1 \notin \Nm_{E/L}(E^\times)$. Suppose that $x\in L$ with $\Nm_{E/L}(x) = -1$. Then $\Nm_{E/\Q}(x) = 1$, and we can write $x = \zeta v$ with $\zeta$ a root of unity and $v \in L^\times$ \cite[Prop.~6.7]{JM}. Consequently, $-1 = \Nm_{E/L}(\zeta)\cdot\Nm_{E/L}(v) = v^2$. But $-1$ is not a square in $L$. \qed
\end{proof}

In the remaining cases, our base field will be $F = \Q(i)$. For each rational prime $p$, we write $p\Z[i] = \mathfrak p_p^2$, $\mathfrak p_p\mathfrak p_p'$ or $\mathfrak p_p = (p)$ for the cases where $p$ is ramified, split or inert in $\Z[i]$, respectively.

\begin{theorem}
\label{thm:res3} 
Let $\Q(i) \subset L \subset E$, with $[E:\Q(i)] = 4 $ and $[E:L] = 2$. Let $\G(E/L)=\langle \sigma \rangle$. Any cyclic division algebra $\left(E/L, \sigma, \gamma \right)$ under these assumptions satisfies $\Nm_{\Z[i]}(\disc({\cal O}_{\nat}/\Z[i])) \ge 2^4\cdot 17^6$, with quality attained for $L = \Q(i,\alpha),$ $ E = \Q(i,\sqrt{\alpha})$ and {$\gamma=1+i$, with $\alpha$ a root of $f(X)=X^2+X-i$}.
\end{theorem}

\begin{proof}:

\noindent\emph{\underline{Step 1.}}   
We start by finding the smallest possible discriminant over $\Q(i)\\
$ for cyclic extensions of degree 4. Let $E$ any be any such extension. By the existence theorem of Class Field Theory, we know that $E$ is contained in the ray class field $ H_{\mathfrak m}$, where $\mathfrak m$ is an admissible modulus for the extension $E/\Q(i)$. The smallest ray class field will have conductor $\mathfrak f = \mathfrak f_{E/\Q(i)}$, which, since $\Q(i)$ is a complex field, will be an ideal of $\Z[i]$. 

The restriction of the Artin map \eqref{eqn:artin} to $I_{\mathfrak f}$ gives us a canonical isomorphism $\G(H_{\mathfrak f}/\Q(i)) \simeq I_{\mathfrak f}/P_{\mathfrak f} = C_{\mathfrak f}$ (cf. \eqref{eqn:artiniso}), which implies $[H_{\mathfrak f}:\Q(i)] = |C_{\mathfrak f}|$. Furthermore, since the class group $C_{\Q(i)}$ is trivial and $\Z[i]^{\times}=\langle i \rangle$, by \eqref{eqn:exact} we have the exact sequence
\begin{align*}
0\to \langle i \rangle \to (\Z[i]/{\mathfrak f})^{\times} \to  C_{\mathfrak f} \to 0.
\end{align*}
Thus, $C_{\mathfrak f} \simeq (\Z[i]/f)^{\times}/ \hbox{Im} \langle i \rangle$.
The ray class field of conductor 2 is trivial, which is not our case, so  $\mathfrak f$ does not divide 2. The map $\langle i \rangle \to(\Z[i]/\mathfrak f)^{\times} $ is injective, so that $[H_{\mathfrak f}:Q(i)] =  \frac{1}{4}\left|(\Z[i]/\mathfrak f)^{\times}\right|$. Consequently, $$4 = [E:Q(i)] \mid [H_{\mathfrak f}:Q(i)] \Rightarrow 16 \mid \left | (\Z[i]/\mathfrak  f)^{\times}\right | = \Nm(\mathfrak f) \text{ and }  \Nm(\mathfrak f)\ge17 .$$

Fortunately we can find ideals of norm 17 with the required properties. 
We fix $\mathfrak f = 1+4i$ (or $\mathfrak f=1-4i$), 
so that the ray class field of conductor $\mathfrak f$ is precisely $H_{\mathfrak f} = \Q(i, \sqrt[4]{1+4i})$. 
Since $\mathfrak f$ is a prime ideal, all non-trivial characters of $\G(H_{\mathfrak f}/\Q(i))$ have conductor $\mathfrak f$, so that by the conductor-discriminant formula \eqref{eqn:condisc}, $\disc(H_{\mathfrak f}/\Q(i)) = \mathfrak f^3$. The absolute discriminant is $\disc(H_{\mathfrak f}/\Q) = 17^3\cdot 4^4$. 

We choose $L = \Q(i, \sqrt{1+4i})$ and $E = H_{\mathfrak f}$, and prove that this choice yields the smallest possible discriminant of a cyclic extension of degree 4 over $\Q(i)$. To that end, let $\mathfrak m$ be any ideal of $\Z[i]$ of norm different from 17 for which $16 \mid \left |(\Z[i]/\mathfrak m)^\times\right |$, and let $E' \subset H_{\mathfrak m}$ be a subfield with $E'/\Q(i)$ cyclic of degree 4. Assume that $\disc(E'/\Q)\le 17^3 \cdot 4^4$. Since $E'$ has conductor $\mathfrak m$, by the minimality of $\mathfrak m$  the quartic characters of $\G(E'/Q(i))$ have conductor $\mathfrak m$. 
The quadratic character could have smaller conductor, but it cannot be smaller than 3, since $\Q(i)$ admits no ray class field of conductor $\mathfrak m$ with $\Nm(\mathfrak m) < 9$. 
Using the conductor-discriminant formula and as norms in $\Q(i)$ are positive, we have
\begin{align}
\label{condisc}
\begin{split}
	&9 \cdot \Nm(\mathfrak m^2) \le \Nm(\disc(E'/\Q(i))), \\
	\Rightarrow\ &9 \cdot \Nm(\mathfrak m)^2\cdot4^4 \le \disc(E'/\Q)< 17^3 \cdot 4^4, \\
	\Rightarrow\ &\Nm(\mathfrak m) < 23,36\ldots,
\end{split}
\end{align}
and so $\left |(\Z[i]/\mathfrak m)^*\right |<23.$ Since $16 \mid \left|(\Z[i]/\mathfrak m)^\times\right|$, we have $\left|(\Z[i]/\mathfrak m)^\times\right| = 16$. But then necessarily $\Nm(\mathfrak m) = 17$, a contradiction to our assumption.

\medskip

\noindent\emph{\underline{Step 2.}}   
Let $\Or_L = \Z[i,\alpha]$, with $\alpha$ a root of $f(X)=X^2+X-i$. A pair of norm-wise smallest primes in $\Or_L$ is $(\mathfrak p_2, \mathfrak p_5)$, of respective norms 2 and 5. Consequently, $\lambda_{E,L} = \frac{20^2}{17^6} < 1$ (cf. \eqref{eqn:constant}), so that $\Nm_{\Z[i]}(\disc(\Or_{\nat}/\Z[i])$ will achieve the smallest possible value among all central division algebras satisfying the given conditions for a unit non-norm element $\gamma \in \Or_L^{\times}$.  Unfortunately, as we will show next, there is no suitable unit in this case, forcing us to consider other non-norm elements.

\medskip
\noindent\emph{\underline{Step 3.}}  
By the Hasse Norm Theorem, to show that an element in $L^{\times}$ is not a norm, it suffices to show that it is not a local norm at some prime of $E$. We need to produce a unit $\gamma \in \Or_L^\times$ with  $\gamma \notin \Nm_{E/L}(E^{\times})$, and since in an unramified local extension all units are norms, we consider the local extension corresponding to the ramifying prime $\mathfrak f = (1+4i)\Z[i]$. It is not difficult to verify that $\mathfrak f\Or_L=\mathfrak p_{L}^2$, $\mathfrak p_{L}\Or_E=\mathfrak p_{E}^2$ and  $\Nm_{L}(\mathfrak p_{L})=\Nm_{E}(\mathfrak p_{E})={17}$.  

 Let $k=L_{\mathfrak p_L}$ and $K= E_{\mathfrak p_E}$ be the completions of $E$ and $L$ with respect to the discrete  valuations associated to the primes $\mathfrak p_L$ of $L$ and $\mathfrak p_E$ of $E$.  Then   $K/k$ is a totally and tamely ramified cyclic local extension of degree 2, with $\overline{K}=\overline {k} = \F_{17}$. Using \eqref{eqn:nonnorm}, $U_k \cap \Nm_{K/k}(K^{\times})=\langle{ \zeta}_{16}^2~\rangle U_{L}^{(1)}$ and, by Hensel's lemma, an element in $\Or_L^{\times}$ will be a non-norm element for $E/L$ if and only if its image in  $\F_{17}$ is not a square in $\F_{17}^{\times}$.  

Since $\Nm_{L/\Q}(x)=\Nm_{\Q(i)/\Q}(\Nm_{L/\Q(i)}(x))$ and the norm of every unit in $\Z[i]$ is $1$, we have $\Nm_{L/\Q}(\Or_L^{\times})=\{1^2\}$. Consequently, the norm of the image in $\F_{17}$ of every element in $\Or_L^{\times}$ is a square. Over finite fields, this is the case if and only it such an image is itself a square, which for its part implies that every element in $\Or_L^{\times}$ maps to $\langle{\zeta}_{16}^2~\rangle$ and, hence, is in the image of the map $ \Nm_{E/L}$. We conclude that there exists no unit which is a non-norm element for $E/L$, forcing us to use\eqref{eqn:balance} to balance the sizes of $\disc(E/\Q(i))$ and $\gamma$ in \eqref{eqn:disc} in order to achieve minimality of $\Nm(\disc(\Or_{\nat}/\Z[i]))$. 
\medskip

  Since the norm of an element $x\in \Or_L$ is the product of the norms of the prime ideals dividing $x$, it is easy to verify that the smallest possible norm in $\Or_L$ is 4. We set $\gamma = 1+i \in \Or_L$ with $\Nm_L(1+i) = 4$, which yields $D_{E/L}(\gamma) = 17^3 \times 2^2$.
In order to study the possible values of $D_{E'/L'}(\gamma')$, let $\mathfrak m$ be any ideal of $\Z[i]$ of norm different from 17, for which $16 \mid \left |(\Z[i]/\mathfrak m)^\times\right |$, and let $E'$ be a subfield of the ray class field $H_{\mathfrak m}$, with $E'/\Q(i)$ cyclic of degree 4.
 
The smallest possible norms for $\mathfrak m$ are given by 
\begin{alignat*} 
	34 &= \Nm(\mathfrak p_{17}\mathfrak p_2),\quad 49 &&= \Nm(\mathfrak p_7), \quad &&64 = \Nm(\mathfrak p)^6, \\ 
	68 &= \Nm(\mathfrak p_{17}\mathfrak p_2^2),\quad  128 &&= \Nm(1+i)^7, \quad &&\ldots
\end{alignat*} 

Using \eqref{condisc}, we see that $9\cdot \Nm(\mathfrak m)^2  \le \Nm(\disc(E'/\Q(i)))$, so that we only need to consider the cases for which $9\cdot \Nm(\mathfrak m)^2  \le 17^3 \times 2^2 \Rightarrow \Nm(\mathfrak m) \le 46,728...$, \emph{i.e.}, only the case $\mathfrak m = (1+4i)(1+i)$. 

We compute $\Nm(\disc(E'/F))$. The smallest quadratic discriminant for an extension in which both $(1+i)$ and a prime of norm $17$ ramify is $(4+i)(1+i)^2$, corresponding to the extension $L' = \Q(\sqrt{4+i})$. Thus, $E' = \Q(\sqrt[4]{4+i})$, with $\disc(E'/\Q(i))=(4+i)^3(1+i)^4$ and $\Nm(\disc(E'/\Q(i)) = 17^3 \cdot 2^4$. Consequently, for all $\gamma'\in \Or_{L'}$, $D_{E'/L'}(\gamma') \ge 17^3\times 2^4  > 17^3 \times 2^2 = D_{E/L}(\gamma)$. By \eqref{eqn:discsize}, we are done.

 \qed
\end{proof}

\begin{theorem}
\label{thm:res4} 
Let $\Q(i) \subset L \subset E$ with $[E:\Q(i)]=6$ and
$ [E:L] = 3$. Any cyclic division algebra $\left(E/L, \sigma, \gamma\right)$ satisfies $\Nm(\disc(\Or_{\nat}/\Z[i]))\ge3^{18}\cdot 13^{12}$. The lower bound is achieved for $L = \Q(\zeta_{12})$, $E = L(\beta)$ and  $\gamma=\frac{1+i\sqrt 3}{2}\in \Or_L^{\times}$, where $\zeta_{12}$ is a primitive $12^\text{th}$ root of unity and $\beta$ a root of $f(X) = X^3-(1+i)X^2+5iX-(1+4i)$.  
\end{theorem}

\begin{proof}:

\noindent\emph{\underline{Step 1.}} 
We start by finding the smallest possible discriminant over $\Z[i]$ for cyclic extensions of degree 6. We denote by $L_2$, $L_3$ and $E = L_2L_3$ cyclic extensions of degree 2 and 3 over $F = \Q(i)$ and their compositum, respectively. Using \eqref{eqn:exact}, \eqref{eqn:condisc} and arguments similar to those used in Theorem~\ref{thm:res3}, we deduce that the smallest possible cubic discriminant is $\mathfrak p_{13}^2 = (2-3i)^2$, corresponding to the extension $\Q(i, \beta)/\Q(i)$ \cite[Table p. 883, row 32]{BeMOl}, where $\beta$ is a root of the polynomial $f(X) = X^3-(1+i)X^2+5iX+(-1-4i)$.

Since we want to minimize $\disc(L_2 L_3/\Q(i))$, the use of \eqref{eqn:exact}, \eqref{eqn:condisc} requires that we consider separately the cases where $\mathfrak m$ is and is not relatively prime to $\mathfrak p_{13}$, and check which case yields a smaller value for this discriminant.
\begin{itemize}
	\item[i)] $(\mathfrak m, \mathfrak p_{13}) = 1$. The smallest possible discriminant corresponds to the extension $L_2=\Q(i)(\sqrt{-3}) =\Q(\zeta_{12})$ of $\Q(i)$, with $\disc(L_2L_3/\Q(i)) = 3^3\mathfrak p_{13}^4$, of norm $\Nm(\disc(L_2L_3/\Q(i))) = 13^4\cdot 3^6$.

	\item[ii)] $(\mathfrak m, \mathfrak p_{13}) \neq 1$. If $\disc(L_2'/\Q(i))=\mathfrak p_{13}\times \mathfrak a$ for some ideal $\mathfrak a \ne(1)$, the best possibility is $\mathfrak p_{13}\mathfrak p_2^2$, corresponding to the extension $L_2'=\Q(i)(\sqrt{3-2i})$ of $\Q(i)$ with discriminant ideal $\mathfrak p_{13}\times \mathfrak p_2^2$. This choice yields\footnote{The factor $\mathfrak p_{13}^5$ comes from the fact that $E/L_3$ is tamely ramified and, thus, has  as discriminant the prime $\mathfrak q_{13}$ lying above 13 in $L_3$.} $\disc(L_2'L_3/\Q(i)) = \mathfrak p_{13}^5\times \mathfrak p_2^6$, with $\Nm(\disc(L_2'L_3/\Q(i))) = 13^5\cdot 2^6 > 13^4\cdot 3^6$.
\end{itemize}

We conclude that the smallest possible discriminant over $\Z[i]$ for cyclic extensions of degree 6 is $3^6 \cdot 13^4$, achieved in the extension $E = L_2L_3$. The involved rings of integers are $\Or_{2} = \Z\left [i, \frac{i+\sqrt{3}}{2}\right]$
and $ \Or_{3} = \Z[i, \beta]$. 
The discriminants of the extensions involved are summarized in Table~\ref{tab:disc} below. 

\begin{table}[H]
\centering
\scalebox{0.98}{
\begin{tabular}{l|ll|ll|ll|ll}
 & $\disc(\cdot/\Q(i))$ & $\Nm_{\Z[i]}$ & $\disc(\cdot/L_2)$ & $\Nm_{\Or_2}$ & $\disc(\cdot/L_3)$ & $\Nm_{\Or_3}$  \\
 \hline
 $E$ & $\mathfrak p_3^3\mathfrak p_{13}^4$ & $3^6\cdot 13^4$ & $\mathfrak q_{13}^2\mathfrak q_{13}'^2$ & $13^4$ &  $\mathfrak s_3$ & $3^6$ \\
 $L_3$ & $\mathfrak p_{13}^2$ & $13^2$ & &  & & \\
 $L_2$ & $\mathfrak p_{3}$ & $3^2$ & &  & &
\end{tabular}}
\caption{Relative discriminants of the field extensions involved.}
\label{tab:disc}
\end{table}
 
\medskip
\noindent\emph{\underline{Step 2.}}  
A pair of smallest primes in $\Or_2$ is $(\mathfrak q_2, \mathfrak q_3)$ of norms 4 and 9, respectively, where $\mathfrak q_2$ is any prime above 2, and $\mathfrak q_3$ is any prime above 3. Consequently, $\lambda_{E,L} = \frac{4^29^2}{13^4}<1$ (cf. \eqref{eqn:constant}), and  $\Nm_{\Z[i]}(\disc(\Or_{\nat}/\Z[i])$ will achieve its smallest possible value  
for a unit non-norm element $\gamma \in \Or_2^{\times}$, $\gamma \notin  \Nm_{E/L_2}(E^\times)$.

\medskip
\noindent\emph{\underline{Step 3.}} To simplify notation, we set, $L_2 = L$, and $\Or_2 = \mcl$. We prove that the unit $\gamma = \frac{1+i\sqrt 3}{2}\in \Or_L^{\times}$ satisfies $\gamma \notin  \Nm_{E/L}(E^\times)$.

The prime $\mathfrak q_{13}$ ramifies in the extension $E/L$. Let $\mathfrak t_{13}$ be a prime of $E$ extending $\mathfrak q_{13}$, and $ k = L_{\mathfrak q_{13}}$, $K = E_{\mathfrak t_{13}}$ be the completions of $L$ and $E$ with respect to the corresponding valuations, with $|\overline {k}|=|\overline{K}|=13$.
 
The local extension  $K/k$ is a totally and tamely ramified extension of degree 3. Since the image of the unit $\gamma = \frac{1+i\sqrt 3}{2}$ in the residue field $\F_{13}$ is 4, which has multiplicative order 6, we deduce that ${\gamma}\not \in \langle ~ \overline{\zeta}_{12}^3 \rangle U_{L}^{(1)}$. By \eqref{eqn:nonnorm} and Hasse's Norm Theorem, the theorem follows.
\qed
\end{proof}

\begin{theorem}
\label{thm:res5}
Let $\Q(i) \subset L \subset E$ with $[E:\Q(i)]=6$ and $ [E:L]=2$, and let $(E/L,\sigma,\gamma)$ be a cyclic division algebra. Then, 
$\Nm(\disc(\Or_{\nat}/\Z[i]))\ge2^{6}\cdot 3^{12}\cdot 13^{8}$ with equality for $\gamma = 1+i$ and $E = LL_2$, where $L_2 = \Q(i, i\sqrt{3}) = \Q(\zeta_{12})$ and $L = \Q(i, \beta)$ with $\beta$ a root of the polynomial $f(X) = X^3-(1+i)X^2+5iX-(1+4i)$.    
\end{theorem}

\begin{proof}:

\noindent\emph{\underline{Step 1, 2.}}  
Let $E$ and $L$ be as in the statement of the theorem. By Theorem~\ref{thm:res4}, the same choice of field $E$ ensures the minimality of the discriminant $\disc(E/\Q(i)) = 3^6\cdot 13^4$ among all possible discriminants of cyclic sextic extensions over $\Q(i)$.  
 
Let $L = \Q(i,\beta)$. A pair of smallest primes in $\Or_L$ is $(\mathfrak s_{2}, \mathfrak s_{13})$, of norms $2^3$ and 13, respectively. Consequently, $\lambda_{E,L} = \frac{2^3 13}{3^3} < 1$ (cf. \eqref{eqn:constant}), and the norm $\Nm_{\Z[i]}(\disc(\Or_{\nat}/\Z[i])$ will attain its smallest possible value for a unit non-norm element $\gamma \in \Or_L^{\times}$.

\medskip 

\noindent\emph{\underline{Step 3.}}  
We encounter here the same situation as in Theorem~\ref{thm:res3}. On the one hand, since $\Nm_{L/\Q}(x)=\Nm_{\Q(i)/\Q}(\Nm_{L/\Q(i)}(x))$ and the norm of every unit in $\Z[i]$ is $1$, every element in $\Or_L^{\times}$ maps to a square in the residue field. On the other hand, $[E:L]=2$, and \eqref{eqn:nonnorm} tells us that $\Or_L^{\times}\cap \Nm_{E/L}(E^{\times})$ consists of those elements in $\Or_L^{\times}$ which map to squares in the residue field of any ramifying prime. Consequently, every element in $\Or_L^{\times}$ is in the image of the map $ \Nm_{E/L}$, and there exists no  non-norm unit element for $E/L$. We thus need to balance the sizes of $\disc(E/\Q(i))$ and $\gamma$ in \eqref{eqn:disc} to achieve minimality of $\Nm(\disc(\Or_{\nat}/\Z[i]))$. We observe that this argument holds for any choice of $E$ and $L$, as long as $\Q(i)\subset L$ and $[E:L] = 2$.
 
We fix $\gamma = 1+i$ of norm $2^3$, corresponding to the smallest possible norm in $\Or_L$. Substituting into \eqref{eqn:balance}, we get $D_{E/L}(\gamma) = 3^6\cdot 13^4\cdot 2^3$. We conclude the proof by showing that $D_{E/L}(\gamma) \le D_{E'/L'}(\gamma')$ for any other choice of $E', L'$ and $\gamma'$ under the given assumptions.

  Any possible $E' \ne E$ is of the form $E' = L_2'L$ with $L_2'\ne L_2$ or $E' = L_2'L'$ with $L'\ne L$ ($L_2'$ could be equal to $L_2$) and $[L':\Q(i)] = 3$. In the first case, by the minimality arguments in the choice of $L_2$ and $\gamma$, for all choices of $L_2'$ and $\gamma' \in \Or_L \setminus \Or_L^{\times}$, 
\begin{align*}
	3^6\cdot 13^4\cdot 2^3\le \Nm(\disc(L_2'L/\Q(i))) \cdot \Nm(\gamma),
\end{align*} 
and so $D_{E/L}(\gamma) \le D_{L_2'L/L}(\gamma')$.

Suppose next that $E' = L_2'L'$ with $L'\ne L$ and $[L':\Q(i)] = 3$. As we saw in Step~1 of the proof of Theorem~\ref{thm:res4}, the conductor of a cubic extension $L'\ne L$ of $\Q(i)$ is ideal $\mathfrak m \in \Z[i]$  of norm greater than 13 and such that $|(\Z[i]/\mathfrak m)^\times|$ is a multiple of 12. By the conductor-discriminant formula, the corresponding extension $L'$ will have discriminant $\mathfrak m^2$, and by the minimality in the choice of $L_2$, $\Nm(\disc (L_2L'/\Q(i))=3^6 \Nm(\mathfrak m)^4\le \disc(L_2L_3' /\Q(i))$ for any quadratic extension $L_2'$ of $\Q(i)$. Consequently, 
\begin{align*}
	3^6 \Nm(\mathfrak m)^4\cdot 2\le \disc(L_3'L_2 /\Q(i))\cdot \Nm(\gamma') = D_{E'/L'}(\gamma')
\end{align*}  
for all choices of $\gamma' \in \Or_L'\setminus \Or_L^{\times}$.
Now, 
\begin{align*}
	3^6 \Nm(\mathfrak m)^4\cdot 2\ge 3^6\cdot 13^4\cdot 2^3   \Leftrightarrow  \Nm (\mathfrak m) \ge 13\sqrt 2> 13.
\end{align*}
We conclude that $D_{E/L}(\gamma) \le D_{E'/L'}(\gamma')$ for all possible choices of $E', L'$ and $\gamma'$, and the theorem follows.
 \qed
\end{proof}

\section{Conclusions}

In this article we have introduced the reader to a technique used in multiple-input multiple-output wireless communications known as space--time coding. Within this framework, we have recalled several design criteria which have been derived in order to ensure a good performance of codes constructed from representations of orders in central simple algebras. In particular, we have explained why it is crucial to choose orders with small discriminants as the underlying algebraic structure in order to maximize the coding gain. While maximal orders achieve the minimal discriminant and hence the maximal coding gain among algebraic space--time codes, we have motivated why in practice it may sometimes be favorable to use the so-called natural orders instead. However, one should bare in mind that orthogonal lattices have yet additional benefits such as simple bit labeling and somewhat simpler encoding and decoding, so there is a natural tradeoff between simplicity and coding gain. 

For the base fields $F=\Q$ or $F$ imaginary quadratic (corresponding to the most typical signaling alphabets), and pairs of extension degrees $(n_t,n_r)$ in an asymmetric channel setup, we have computed an explicit number field extension $(E/L)$ and a non-norm element $\gamma \in \mcl\setminus\left\{0\right\}$ giving rise to a cyclic division algebra whose  natural order $\Or_{\nat}$  achieves the minimum discriminant among all cyclic division algebras with the same degree and base field assumptions. This way we have produced explicit space--time codes attaining the optimal coding gain among codes arising from natural orders.

\begin{acknowledgements}
A. Barreal and C. Hollanti are financially supported by the Academy of Finland grants \#276031, \#282938, and \#303819, as well as a grant from the Foundation for Aalto University Science and Technology.

\medskip 
The authors thank Jean Martinet, Ren\'e Schoof, and Bharath Sethuraman for their useful suggestions, and the anonymous reviewers for their valuable comments to improve the quality of the manuscript. 
\end{acknowledgements}


\begin{thebibliography}{99}

\bibitem{TSC}
V. Tarokh, N. Seshadri, and A. R. Calderbank, 
Space--Time Codes for High Data Rate Wireless Communication: Performance Criterion and Code Construction, 
\textit{IEEE Transactions on Information Theory} 44(2), pp. 744--765, 1998.

\bibitem{BR}
J.-C. Belfiore and G. Rekaya, 
Quaternionic Lattices for Space--Time Coding,
\textit{Proceedings of the IEEE Information Theory Workshop}, 2003.

\bibitem{SRS}
B. A. Sethuraman, B. S. Rajan, and V. Shashidhar, 
Full-Diversity, High-Rate Space--Time Block Codes from Division Algebras, 
\textit{IEEE Transactions on Information Theory} 49(10), pp. 2596--2616, 2003.

\bibitem{BRV}
J.-C. Belfiore, G. Rekaya, and E. Viterbo,
The Golden Code: a $2\times 2$ Full-Rate Space--Time Code with Non-vanishing Determinants, 
\textit{IEEE Transactions on Information Theory} 51(4), pp. 1432--1436, 2005.

\bibitem{ORBV}
F.~Oggier, G.~Rekaya, J.-C. Belfiore, and E.~Viterbo, 
Perfect Space--Time Block Codes, 
\textit{IEEE Transactions on Information Theory} 52(9), pp. 3885--3902, 2006.

\bibitem{ESK}
P. Elia, B. A. Sethuraman, and P. V. Kumar, 
Perfect Space--Time Codes for Any Number of Antennas, 
\textit{IEEE Transactions on Information Theory} 53(11), pp. 3853--3868, 2007.

\bibitem{HLL}
C. Hollanti, J. Lahtonen, and H.-f. Lu, 
Maximal Orders in the Design of Dense Space--Time Lattice Codes,
\textit{IEEE Transactions on Information Theory} 54(10), pp. 4493--4510, 2008.

\bibitem{VHLR}
R. Vehkalahti, C. Hollanti, J. Lahtonen, and K. Ranto, 
On the Densest MIMO Lattices From Cyclic Division Algebras, 
\textit{IEEE Transactions on Information Theory} 55(8), pp. 3751--3780, 2009.

\bibitem{HL}
C. Hollanti and H.-f. Lu,
Construction Methods for Asymmetric and Multiblock Space--Time Codes,
\textit{IEEE Transactions on Information Theory} 55(3), pp. 1086--1103, 2009.

\bibitem{Re}
I. Reiner,
Maximal Orders, 
\textit{London Mathematical Society Monographs New Series} 28, 2003.

\bibitem{Veh}
R. Vehkalahti, 
Class Field Theoretic Methods in the Design of Lattice Signal Constellations, 
\textit{TUCS Dissertations} 100, 2008.
 
\bibitem{VHO}  
R. Vehkalahti, C. Hollanti, and F. Oggier, 
Fast-Decodable Asymmetric Space-Time Codes From Division Algebras, 
\emph{IEEE Transactions on Information Theory} 58(4), pp. 2362--2385, 2012.
 
 
\bibitem{CS}
H. Cohen and P. Stevenhagen, 
Computational Class Field theory, 
\textit{Algorithmic Number Theory, MSRI Publications} 44, pp. 497--534, 2008.

\bibitem{JM} 
J. S. Milne,
Algebraic Number Theory (v3.07), 
\textit{Graduate course notes}, 2017, available at www.jmilne.org/math/CourseNotes/.

\bibitem{JM2} 
J. S. Milne, 
Class Field Theory (v4.02), 
\textit{Graduate course notes}, 2013, available at www.jmilne.org/math/CourseNotes/.

\bibitem{HW} 
R. H. Hudson and K. S. Williams, 
The Integers of a Cyclic Quartic Field, 
\textit{Rocky Mountains Journal of Mathematics} 20(1), pp. 145--150, 1990.
  
\bibitem{HSW} 
J.G. Huard, B.K. Spearman, and K. S. Williams, 
Integral Bases for Quartic Fields wit Quadratic Subfields, 
\textit{Carleton-Ottawa Mathematical Lecture Notes Series} 7, pp. 87--102, 1986.

\bibitem{BeMOl}
A.-M. Berg\'e, J. Martinet, and M. Olivier, 
The Computation of Sextic Fields with a Quadratic Subfield,
\textit{Mathematics of Computation} 54(190), pp. 869--884, 1990.





  
 

  





\end{thebibliography}
\end{document}